\documentclass[10pt]{amsart}

\usepackage{amssymb,latexsym,amsmath,epsfig,amsthm} %% Add other packages as necessary
\usepackage{mathrsfs}
\usepackage{setspace}
\usepackage{url}

\newtheorem{theorem}{Theorem}
\newtheorem{lemma}{Lemma}
\newtheorem{proposition}{Proposition}

\theoremstyle{definition}

\newtheorem{conjecture}{Conjecture}

\title{A Series Involving a Product of Four Consecutive Harmonic Numbers}
\author{Wilson J. Chen and Vincent Nguyen}


\begin{document}
\begin{abstract}
    In correspondence with Goldbach, Euler began investigating series of the form
    $\sum_{k \geq 1} k^{-m}\left(1 + 2^{-n} + \cdots + k^{-n}\right)$, which are known today as Euler sums. For the case where $n=1$ and $m \geq 2$, Euler was able to obtain a closed form in terms of zeta values. We use elementary techniques in the spirit of Euler to evaluate the series $\sum_{k \geq 1} \frac{H_k H_{k+1} H_{k+2} H_{k+3}}{k(k+1)(k+2)(k+3)},$ where $H_k := 1 + \frac{1}{2} + \cdots + \frac{1}{k}$ is the $k$th harmonic number, in terms of zeta values. The closed form is a potential counterexample to a conjecture of Furdui and S{\^\i}nt{\u{a}}m{\u{a}}rian. We relate this problem to conjectures regarding irrationality and $\mathbb{Q}$-linear independence of zeta values.
\end{abstract}

\maketitle

\section{Introduction}
Euler originally investigated the series
$$\sum_{k=1}^\infty \frac{1}{k^m}\left(1 + \frac{1}{2^n} + \cdots + \frac{1}{k^n}\right), \quad (m\in \mathbb{N}\setminus\{1\}, \,n \in \mathbb{N})$$
in \cite{Euler1}. Euler and Goldbach were able to determine that
\begin{equation} \label{EulerGoldbach}
    \sum_{k=1}^\infty \frac{H_k}{k^2} = 2\zeta(3) \quad \text{and} \quad \sum_{k=1}^\infty \frac{H_k}{k^3} = \frac{5}{4}\zeta(4)
\end{equation}
respectively \cite[pp. 238–240]{SharpMathAnal}. Here, $H_k := 1 + \frac{1}{2} + \cdots + \frac{1}{k}$ is the $k$th harmonic number and
\begin{equation}\label{ZetaFunctionDef}
    \zeta(s) = \sum_{k=1}^\infty \frac{1}{k^s}, \quad (\Re(s) > 1)
\end{equation}
denotes the classical Riemann zeta function. The sums Euler and Goldbach derived are specific cases of the formula
\begin{equation} \label{OriginalEulerSum}
    \sum_{k=1}^\infty \frac{H_k}{k^n} = \left(\frac{n}{2} + 1\right)\zeta(n+1) - \frac{1}{2}\sum_{k=1}^{n-2} \zeta(n-k)\zeta(k+1), \quad (n \in \mathbb{N}\setminus\{1\}).
\end{equation}
The empty sum that arises at $n=2$ is understood to be nil. We mention that \eqref{OriginalEulerSum} is due to Euler \cite{ExplicitEulerSum, AssociatedFun}. Numerous results pertaining to Euler sums can be found in \cite{ExplicitEulerSum, FlajoletEuler, Sofo, LogIntHarmSeriesBook, Xu,XuEuler, XuEulerSums2, XuEulerSums3, XuEulerSums4}. Euler sums are related to the multiple zeta function, defined as
$$\zeta(s_1, \ldots, s_r) = \sum_{n_1 > \cdots > n_r \geq 1} \frac{1}{n_1^{s_1} \cdots n_r^{s_r}},$$
which has been studied in more recent times for its significance in diverse areas of mathematics \cite{CombAlgMZV,AlgebraMZV, ZhaoMZV}. Note that when $s_1, \ldots, s_r$ are naturals with $s_1>1$, the result is an Euler sum.

Many years later, an undergraduate student at the University of Waterloo named Enrico Au-Yeung rediscovered and conjectured the following value for the following quadratic sum:
\begin{equation} \label{Au-Yeung}
    \sum_{k=1}^\infty \left(\frac{H_k}{k}\right)^2 = \frac{17}{4}\zeta(4).
\end{equation}
David Borwein and Jonathan Borwein would eventually verify the result using techniques involving Fourier series and Parseval's theorem in \cite{Borwein}. A more elementary proof is provided by V{\u{a}}lean and Furdui in \cite{FurduiValean}. In deference to Au-Yeung, the series is called the quadratic series of Au-Yeung. The cubic version of Au-Yeung's series was computed in \cite{Valean3} to be
$$\sum_{k=1}^\infty \left(\frac{H_k}{k}\right)^3 = \frac{93}{16}\zeta(6) - \frac{5}{2}\zeta^2(3).$$
One would expect that the quartic analogue can also be expressed zeta values. However, it was determined in \cite{EulerStirling} that
$$\sum_{k=1}^\infty \left(\frac{H_k}{k}\right)^4 = \frac{13559}{144}\zeta(8) - 92\zeta(3)\zeta(5) -2\zeta(2)\zeta^2(3) + 26\texttt{S}_{2,6}$$
where
$$\texttt{S}_{m,n} := \sum_{k=1}^\infty \frac{1}{k^n}\left(1 + \frac{1}{2^m} + \cdots + \frac{1}{k^m}\right), \quad (m \in \mathbb{N}, n\in \mathbb{N}\setminus\{1\}),$$
which is an Euler sum not known to be reducible to zeta values when $m=2$ and $n=6$.

Furdui and S{\^\i}nt{\u{a}}m{\u{a}}rian considered a different type of series involving products of harmonic numbers; rather than raising $\frac{H_k}{k}$ to an integer power, they considered series involving the consecutive product $\frac{H_k H_{k+1} \cdots H_{k+q}}{k(k+1)\cdots(k+q)}$ for some integer $q \geq 1$. They were able to compute the following series in closed form:
\begin{equation} \label{Consecutive2}
    \sum_{k=1}^\infty \frac{H_k H_{k+1}}{k(k+1)} = \zeta(2) + 2\zeta(3)
\end{equation}
and
\begin{equation} \label{Consecutive3}
    \sum_{k=1}^\infty \frac{H_k H_{k+1} H_{k+2}}{k(k+1)(k+2)} = -\frac{1}{2}\zeta(2) + \frac{5}{4}\zeta(3) + \frac{5}{8}\zeta(4).
\end{equation}
A generalization of \eqref{Consecutive2} is provided in \cite{FurduiProduct2Series}:
\begin{align*}
    \sum_{k=1}^\infty \frac{H_k H_{k+\ell}}{k(k+\ell)}
    =& \sum_{i=0}^{\ell-1} (-1)^i \binom{\ell-1}{i}\frac{1}{(i+1)^3}\left(\frac{3}{i+1} - \frac{2}{\ell}\right)\\
    &+ \frac{\pi^2}{6}\frac{H_\ell}{\ell} + \frac{2\zeta(3)}{\ell} - \frac{1}{\ell}\sum_{j=1}^{\ell} \frac{H_j}{j^2},\quad (\ell \geq 1).
\end{align*}
In \cite[p. 119]{SharpMathAnal}, \eqref{Consecutive2} is computed and \eqref{Consecutive3} is left as an open problem. In \cite{FurduiGazetaMat}, \eqref{Consecutive3} is computed. In the very same article, Furdui and S{\^\i}nt{\u{a}}m{\u{a}}rian make the following conjecture:
\begin{conjecture}\label{FurduiSintConjecture}
    If $q \geq 1$ is an integer, then
    $$\sum_{k=1}^\infty \frac{H_k H_{k+1}\cdots H_{k+q}}{k(k+1)\cdots (k+q)} = \alpha_{q,2}\zeta(2) + \alpha_{q, 3}\zeta(3) + \cdots + \alpha_{q, q+2} \zeta(q+2)$$
    where $\alpha_{q,2}, \alpha_{q,3}, \ldots, \alpha_{q, q+2}$ are rationals.
\end{conjecture}
We note that Conjecture \ref{FurduiSintConjecture} possibly fails when $q=3$, as the series might not be expressible as a linear combination of zeta values $\zeta(2), \zeta(3), \zeta(4), \zeta(5)$ over the rationals. To illustrate this, we use elementary methods to directly show that
$$\sum_{k=1}^\infty \frac{H_k H_{k+1} H_{k+2} H_{k+3}}{k(k+1)(k+2)(k+3)} = -\frac{4}{9} \zeta (2)-\frac{1}{6}\zeta (2) \zeta (3)-\frac{7}{24} \zeta (3)+\frac{191}{144} \zeta (4)+\frac{1}{2}\zeta (5).$$

\section{Proof of the Main Result}
We require some lemmas.
\begin{lemma} \label{ExpandConsecutiveProd}
    Let
    $$\mathscr{H}_k := \frac{H_k H_{k+1} H_{k+2} H_{k+3}}{k(k+1)(k+2)(k+3)}.$$
    Then we have
    \begin{align*}
        \mathscr{H}_k
        =& \frac{11}{24 (k+1)^2} - \frac{131}{72 (k+2)^2} + \frac{11}{12(k+3)^2}- \frac{5}{9 (k+1)^3} + \frac{109}{72 (k+2)^3} \\
        &- \frac{41}{36 (k+3)^3} + \frac{1}{6 (k+1)^4} - \frac{5}{12(k+2)^4} + \frac{11}{36(k+3)^4} + \frac{11 H_k}{24 k} \\
        &-\frac{41 H_{k+1}}{18 (k+1)} + \frac{197 H_{k+2}}{72 (k+2)} - \frac{11 H_{k+3}}{12 (k+3)} + \frac{73 H_k^2}{72 k} - \frac{11 H_{k+1}^2}{3(k+1)} \\
        &+\frac{97 H_{k+2}^2}{24 (k+2)} - \frac{25 H_{k+3}^2}{18 (k+3)} + \frac{13 H_k^3}{18 k} - \frac{13 H_{k+1}^3}{6 (k+1)} + \frac{13 H_{k+2}^3}{6 (k+2)} -\frac{13 H_{k+3}^3}{18 (k+3)} \\
        &+ \frac{H_k^4}{6 k} - \frac{H_{k+1}^4}{2 (k+1)} + \frac{H_{k+2}^4}{2 (k+2)} - \frac{H_{k+3}^4}{6 (k+3)} +\frac{9 H_{k+1}}{4 (k+1)^2} - \frac{29 H_{k+2}}{6 (k+2)^2} \\
        &+ \frac{91 H_{k+3}}{36 (k+3)^2} - \frac{2 H_{k+1}}{3 (k+1)^3} + \frac{13 H_{k+2}}{6 (k+2)^3} - \frac{4 H_{k+3}}{3 (k+3)^3} + \frac{H_{k+3}}{6 (k+3)^4} \\
        &+ \frac{9 H_{k+1}^2}{4 (k+1)^2} - \frac{15 H_{k+2}^2}{4 (k+2)^2} + \frac{7 H_{k+3}^2}{4 (k+3)^2} +\frac{H_{k+2}^2}{2(k+2)^3} - \frac{H_{k+3}^2}{2(k+3)^3} \\
        &+ \frac{H_{k+1}^3}{2 (k+1)^2} - \frac{H_{k+2}^3}{(k+2)^2} + \frac{H_{k+3}^3}{2 (k+3)^2}.
    \end{align*}
\end{lemma}
The proof relies on repeatedly applying partial fraction decomposition and the well-known relation $H_{k+q} = H_{k+p} + \frac{1}{k + p+1} + \frac{1}{k + p + 2} + \cdots + \frac{1}{k + q}$ where $q > p \geq 0$ are integers.
\begin{proof}
    Our goal is to expand $\mathscr{H}_k$ to express it as a rational linear combination of terms of the form $$\frac{H_{k+m}^\ell}{(k+m)^n}, \quad (\ell, m \in \mathbb{N}\cup\{0\}, n \in \mathbb{N}).$$
    We begin by applying partial fraction decomposition:
    \begin{align*}
        \mathscr{H}_k
        =& H_k H_{k+1} H_{k+2} H_{k+3}\left(\frac{1}{6 k} - \frac{1}{2 (k+1)} + \frac{1}{2 (k+2)} - \frac{1}{6 (k+3)}\right)\\
        =& \frac{H_k H_{k+1} H_{k+2} H_{k+3}}{6k} - \frac{H_k H_{k+1} H_{k+2} H_{k+3}}{2 (k+1)} + \frac{H_k H_{k+1} H_{k+2} H_{k+3}}{2 (k+2)}\\
        &- \frac{H_k H_{k+1} H_{k+2} H_{k+3}}{6 (k+3)}.
    \end{align*}
    We direct our attention to $\frac{H_k H_{k+1} H_{k+2} H_{k+3}}{6k}$:
    \begin{align*}
        &\frac{H_k H_{k+1} H_{k+2} H_{k+3}}{6k}\\
        =& \frac{H_k\left(H_k + \frac{1}{k+1}\right)\left(H_k + \frac{1}{k+1} + \frac{1}{k+2}\right)\left(H_k + \frac{1}{k+1} + \frac{1}{k+2} + \frac{1}{k+3}\right)}{6k}\\
        =& \frac{H_k^4}{6 k} + \frac{H_k^3}{2 k (k+1)} +  \frac{H_k^3}{3 k (k+2)} +  \frac{H_k^3}{6 k (k+3)} + \frac{H_k^2}{2k (k+1)^2} + \frac{H_k^2}{6k (k+2)^2}\\
        &+ \frac{2 H_k^2}{3 k (k+1)(k+2)} + \frac{H_k^2}{3 k (k+1)(k+3)} + \frac{H_k^2}{6 k (k+2)(k+3)}\\
        &+ \frac{H_k}{6 k (k+1)^3} + \frac{H_k}{6k (k+1)(k+2)^2} + \frac{H_k}{3k (k+1)^2 (k+2)}\\
        & +\frac{H_k}{6k (k+1)^2 (k+3)}+ \frac{H_k}{6 k(k+1)(k+2)(k+3)}.
    \end{align*}
    We again apply partial fraction decomposition to every single term:
    \begin{align*}
        \frac{H_k H_{k+1} H_{k+2} H_{k+3}}{6k}
        =& \frac{11 H_k}{24 k}-\frac{11 H_k}{24 (k+1)} +\frac{H_k}{24 (k+2)} -\frac{H_k}{24 (k+3)}\\
        &-\frac{7 H_k}{12 (k+1)^2}+\frac{H_k}{12 (k+2)^2} -\frac{H_k}{6 (k+1)^3} +\frac{73 H_k^2}{72 k}\\
        &-\frac{4 H_k^2}{3 (k+1)}+\frac{5 H_k^2}{24 (k+2)} +\frac{H_k^2}{9 (k+3)} -\frac{H_k^2}{2 (k+1)^2}\\
        &-\frac{H_k^2}{12 (k+2)^2}+ \frac{13 H_k^3}{18 k}-\frac{H_k^3}{2 (k+1)}-\frac{H_k^3}{6 (k+2)}\\
        &-\frac{H_k^3}{18 (k+3)} + \frac{H_k^4}{6k}.
    \end{align*}
    If there is a term of the form $\frac{H^\ell_{k+p}}{(k+q)^n}$ where $p=q$ then we are done. If $p>q$, we substitute $H_{k+p}=H_{k+q} + \frac{1}{k+q+1} + \frac{1}{k+q+2} + \cdots + \frac{1}{k+p}$. If $p<q$, we substitute $H_{k+p}=H_{k+q} - \frac{1}{k+q} - \frac{1}{k+q-1} - \cdots - \frac{1}{k+p+1}$. After this substitution, expand the new expression and apply partial fraction decomposition. The result may again contain some terms where $p\neq q$, in which case we repeat our substitution of $H_{k+p}$. We achieve our goal by iterating this process. We apply the same principles to $\frac{H_k H_{k+1} H_{k+2} H_{k+3}}{2 (k+1)},$ $ \frac{H_k H_{k+1} H_{k+2} H_{k+3}}{2 (k+2)},$ and $\frac{H_k H_{k+1} H_{k+2} H_{k+3}}{6 (k+3)}$, eventually giving us the desired result. The tedious calculation details are omitted for the sake of brevity.
\end{proof}
We also require the following lemma:
\begin{lemma} \label{ReIndexZetaLemma}
    The following equalities hold:
    \begin{align*}
        &\textit{(a)}\quad \mathcal{S}_1 := \sum_{k=1}^\infty \left(\frac{11}{2 (k+1)^2} - \frac{131}{6 (k+2)^2} + \frac{11}{(k+3)^2}\right) = \frac{491}{72} - \frac{16}{3}\zeta(2),\\
        &\textit{(b)}\quad \mathcal{S}_2 := \sum_{k=1}^\infty \left(- \frac{5}{(k+1)^3} + \frac{109}{8 (k+2)^3} -\frac{41}{4 (k+3)^3}\right) = \frac{2735}{1728}-\frac{13}{8}\zeta(3),\\
        &\textit{(c)}\quad \mathcal{S}_3 := \sum_{k=1}^\infty \left(\frac{1}{(k+1)^4} - \frac{5}{2 (k+2)^4} + \frac{11}{6 (k+3)^4}\right) = -\frac{611}{1944} + \frac{1}{3}\zeta(4).
    \end{align*}
\end{lemma}
\begin{proof}
    \textit{(a)} By re-indexing, we have
    \begin{align*}
        \mathcal{S}_1
        =& \frac{11}{2}\sum_{k=2}^\infty \frac{1}{k^2} - \frac{131}{6}\sum_{k=3}^\infty \frac{1}{k^2} + 11\sum_{k=4}^\infty \frac{1}{k^2}\\
        =& \frac{11}{2}\left(-\frac{1}{1^2} + \sum_{k=1}^\infty \frac{1}{k^2}\right) - \frac{131}{6}\left(-\frac{1}{1^2} - \frac{1}{2^2} + \sum_{k=1}^\infty \frac{1}{k^2}\right)\\
        &+ 11\left(-\frac{1}{1^2} - \frac{1}{2^2} - \frac{1}{3^2} + \sum_{k=1}^\infty \frac{1}{k^2}\right)\\
        =& \frac{491}{72} - \frac{16}{3}\sum_{k=1}^\infty \frac{1}{k^2}\\
        =& \frac{491}{72} - \frac{16}{3}\zeta(2).
    \end{align*}

    \noindent
    \textit{(b)}, \textit{(c)} The calculation details of $\mathcal{S}_2$ and $\mathcal{S}_3$ are similar to that of $\mathcal{S}_1$. Thus, we omit these details for the sake of brevity.
\end{proof}

\begin{lemma} \label{TelescopingLemma}
The following equalities hold:
\begin{align*}
    &\textit{(a)}\quad \mathcal{S}_4 := \sum_{k=1}^\infty \left(\frac{11 H_k}{4 k} -\frac{41 H_{k+1}}{3 (k+1)} + \frac{197 H_{k+2}}{12 (k+2)} - \frac{11 H_{k+3}}{2 (k+3)}\right) = -\frac{299}{144},\\
    &\textit{(b)}\quad \mathcal{S}_5 := \sum_{k=1}^\infty \left(\frac{73 H_k^2}{24 k} - \frac{11 H_{k+1}^2}{k+1} +\frac{97 H_{k+2}^2}{8 (k+2)} - \frac{25 H_{k+3}^2}{6 (k+3)}\right) = -\frac{6445}{5184},\\
    &\textit{(c)}\quad \mathcal{S}_6 := \sum_{k=1}^\infty \left(\frac{H_k^3}{3 k} - \frac{H_{k+1}^3}{k+1} + \frac{H_{k+2}^3}{k+2} -\frac{H_{k+3}^3}{3 (k+3)}\right) = -\frac{26}{243},\\
    &\textit{(d)}\quad \mathcal{S}_7 := \sum_{k=1}^\infty \left(\frac{H_k^4}{3 k} - \frac{H_{k+1}^4}{k+1} + \frac{H_{k+2}^4}{k+2} - \frac{H_{k+3}^4}{3 (k+3)}\right) = -\frac{577}{5832}.
\end{align*}
\end{lemma}
\begin{proof}
    \textit{(a)} We take advantage of telescoping:
    \begin{align*}
        \mathcal{S}_4
        =& \frac{11}{4}\sum_{k=1}^\infty \left(\frac{H_k}{k} - \frac{H_{k+1}}{k+1}\right) - \frac{131}{12}\sum_{k=1}^\infty \left( \frac{H_{k+1}}{k+1} -  \frac{H_{k+2}}{k+2}\right) \\
        &+ \frac{11}{2}\sum_{k=1}^\infty \left(\frac{H_{k+2}}{k+2} - \frac{H_{k+3}}{k+3}\right)\\
        =& \frac{11}{4} \cdot \frac{H_1}{1} - \frac{131}{12} \cdot \frac{H_2}{2} + \frac{11}{2} \cdot \frac{H_3}{3}\\
        =& -\frac{299}{144}.
    \end{align*}

    \noindent
    \textit{(b)} We again take advantage of telescoping:
    \begin{align*}
        \mathcal{S}_5
        =& \frac{73}{24}\sum_{k=1}^\infty \left(\frac{H_k^2}{k} - \frac{H_{k+1}^2}{k+1}\right) - \frac{191}{24}\sum_{k=1}^\infty \left(\frac{H_{k+1}^2}{k+1} - \frac{H_{k+2}^2}{k+2}\right) \\
        &+ \frac{25}{6}\sum_{k=1}^\infty \left(\frac{H_{k+2}^2}{k+2} - \frac{H_{k+3}^2}{k+3}\right)\\
        =& \frac{73}{24} \cdot \frac{H_1^2}{1} - \frac{191}{24} \cdot \frac{H_2^2}{2} + \frac{25}{6} \cdot \frac{H_3^2}{3}\\
        =& -\frac{6445}{5184}.
    \end{align*}

    \noindent
    \textit{(c)} Telescoping in $\mathcal{S}_6$ is more clear:
    \begin{align*}
        \mathcal{S}_6
        =& \frac{1}{3}\sum_{k=1}^\infty \left(\frac{H_k^3}{k} - \frac{H_{k+3}^3}{k+3}\right) - \sum_{k=1}^\infty \left(\frac{H_{k+1}^3}{k+1} - \frac{H_{k+2}^3}{k+2}\right)\\
        =& \frac{1}{3}\left(\frac{H_1^3}{1}+\frac{H_2^3}{2}+\frac{H_3^3}{3}\right) - \frac{H_2^3}{2}\\
        =& -\frac{26}{243}.
    \end{align*}

    \noindent
    \textit{(d)} The calculations for $\mathcal{S}_7$ is nearly identical to that of $\mathcal{S}_6$.
\end{proof}

\begin{lemma} \label{ReIndexEulerSummLemma}
The following equalities hold:
\begin{align*}
    &\textit{(a)}\quad \mathcal{S}_8 := \sum_{k=1}^\infty \left(\frac{9 H_{k+1}}{2 (k+1)^2} - \frac{29 H_{k+2}}{3 (k+2)^2} + \frac{91 H_{k+3}}{18 (k+3)^2}\right) = \frac{3151}{3888} -\frac{2}{9}\zeta(3),\\
    &\textit{(b)}\quad \mathcal{S}_9 := \sum_{k=1}^\infty \left(- \frac{2 H_{k+1}}{(k+1)^3} + \frac{13 H_{k+2}}{2 (k+2)^3} - \frac{4 H_{k+3}}{(k+3)^3}\right) = -\frac{1807}{2592}+\frac{5}{8}\zeta(4),\\
    &\textit{(c)}\quad \mathcal{S}_{10} := \sum_{k=1}^\infty \frac{H_{k+3}}{(k+3)^4} = - \frac{8681}{7776} -\zeta(2) \zeta(3) + 3\zeta(5),\\
    &\textit{(d)}\quad \mathcal{S}_{11} := \sum_{k=1}^\infty \left(\frac{9 H_{k+1}^2}{(k+1)^2} - \frac{15 H_{k+2}^2}{(k+2)^2} + \frac{7 H_{k+3}^2}{(k+3)^2}\right) = \frac{287}{324} + \frac{17}{4}\zeta(4).
\end{align*}
\end{lemma}
\begin{proof}
    \textit{(a)} By re-indexing, we have
    \begin{align*}
        \mathcal{S}_8
        \overset{\phantom{\eqref{EulerGoldbach}, \eqref{OriginalEulerSum}}}{=}& \frac{9}{2}\sum_{k=2}^\infty \frac{H_k}{k^2} - \frac{29}{3}\sum_{k=3}^\infty \frac{H_k}{k^2} + \frac{91}{18}\sum_{k=4}^\infty \frac{H_k}{k^2}\\
        \overset{\phantom{\eqref{EulerGoldbach}, \eqref{OriginalEulerSum}}}{=}& \frac{9}{2}\left(-\frac{H_1}{1^2} + \sum_{k=1}^\infty \frac{H_k}{k^2}\right) - \frac{29}{3}\left(-\frac{H_1}{1^2} - \frac{H_2}{2^2} + \sum_{k=1}^\infty \frac{H_k}{k^2}\right)\\
        &+ \frac{91}{18}\left(-\frac{H_1}{1^2} - \frac{H_2}{2^2} - \frac{H_3}{3^2} + \sum_{k=1}^\infty \frac{H_k}{k^2}\right)\\
        \overset{\phantom{\eqref{EulerGoldbach}, \eqref{OriginalEulerSum}}}{=}& \frac{3151}{3888} - \frac{1}{9}\sum_{k=1}^\infty \frac{H_k}{k^2}\\
        \overset{\eqref{EulerGoldbach}, \eqref{OriginalEulerSum}}{=}& \frac{3151}{3888} - \frac{2}{9}\zeta(3).
    \end{align*}

    \noindent
    \textit{(b)} The calculations for $\mathcal{S}_9$ are very similar to that of $\mathcal{S}_8$.

    \noindent
    \textit{(c)} Re-index $\mathcal{S}_{10}$ and use \eqref{OriginalEulerSum}.

    \noindent
    \textit{(d)} The calculations for $\mathcal{S}_{11}$ are very similar to that of $\mathcal{S}_8$. However, instead of applying the formula for Euler sums \eqref{OriginalEulerSum}, we plug in the quadratic series of Au-Yeung \eqref{Au-Yeung}.
\end{proof}

\begin{lemma} \label{PowerEulerSumLemma}
    The following equalities hold:
    \begin{align*}
        &\textit{(a)}\quad \mathcal{S}_{12} := \sum_{k=1}^\infty \left(\frac{H_{k+2}^2}{(k+2)^3} - \frac{H_{k+3}^2}{(k+3)^3}\right) = \frac{121}{972},\\
        &\textit{(b)}\quad \mathcal{S}_{13} := \sum_{k=1}^\infty \left(\frac{H_{k+1}^3}{2 (k+1)^2} - \frac{H_{k+2}^3}{(k+2)^2} + \frac{H_{k+3}^3}{2 (k+3)^2}\right) = \frac{1237}{15552}.
    \end{align*}
\end{lemma}
\begin{proof}
    \textit{(a)} The series telescopes.

    \noindent
    \textit{(b)} We have
    \begin{align*}
        \mathcal{S}_{13}
        =& \frac{1}{2}\sum_{k=1}^\infty \left(\frac{H_{k+1}^3}{(k+1)^2} - \frac{H_{k+2}^3}{(k+2)^2}\right) -\frac{1}{2}\sum_{k=1}^\infty \left( \frac{H_{k+2}^3}{(k+2)^2} - \frac{H_{k+3}^3}{(k+3)^2}\right).
    \end{align*}
    The remaining two series telescope.
\end{proof}

We are ready to prove the main result.
\begin{theorem} \label{MainProposition}
    The following equality holds:
    $$\sum_{k=1}^\infty \frac{H_k H_{k+1} H_{k+2} H_{k+3}}{k(k+1)(k+2)(k+3)} = -\frac{4}{9} \zeta (2)-\frac{1}{6}\zeta (2) \zeta (3)-\frac{7}{24} \zeta (3)+\frac{191}{144} \zeta (4)+\frac{1}{2}\zeta (5).$$
\end{theorem}
\begin{proof}
    Let
    $$\mathfrak{S} := \sum_{k=1}^\infty \frac{H_k H_{k+1} H_{k+2} H_{k+3}}{k(k+1)(k+2)(k+3)}.$$
    Recall $\mathcal{S}_1, \mathcal{S}_2, \ldots, \mathcal{S}_{13}$ defined in Lemmas \ref{ReIndexZetaLemma}, \ref{TelescopingLemma}, \ref{ReIndexEulerSummLemma}, and \ref{PowerEulerSumLemma}. Using Lemma \ref{ExpandConsecutiveProd}, we expand $\mathfrak{S}$:
    \begin{align*}
        \mathfrak{S}
        =& \frac{1}{12}\mathcal{S}_1 + \frac{1}{9}\mathcal{S}_2 + \frac{1}{6}\mathcal{S}_3 + \frac{1}{6}\mathcal{S}_4 + \frac{1}{3}\mathcal{S}_5 + \frac{13}{6}\mathcal{S}_6 + \frac{1}{2}\mathcal{S}_7 + \frac{1}{2}\mathcal{S}_8 + \frac{1}{3}\mathcal{S}_9 + \frac{1}{6}\mathcal{S}_{10}\\
        &+ \frac{1}{4}\mathcal{S}_{11} + \frac{1}{2}\mathcal{S}_{12} + \mathcal{S}_{13}.
    \end{align*}
    Apply Lemmas \ref{ReIndexZetaLemma}, \ref{TelescopingLemma}, \ref{ReIndexEulerSummLemma}, and \ref{PowerEulerSumLemma}; this concludes the proof of the theorem.
\end{proof}

The closed form in Theorem \ref{MainProposition} is a potential counterexample to Conjecture \ref{FurduiSintConjecture}. Using Theorem \ref{MainProposition}, express the series as
\begin{equation} \label{Zeta2Zeta3Series}
    \begin{split}
        \sum_{k=1}^\infty \frac{H_k H_{k+1} H_{k+2} H_{k+3}}{k(k+1)(k+2)(k+3)} 
        &=  -\left(\frac{4}{9} + \frac{1}{6}\zeta(3)\right)\zeta (2) -\frac{7}{24} \zeta (3)+\frac{191}{144} \zeta (4)+\frac{1}{2}\zeta (5)\\
        &= -\frac{4}{9}\zeta(2) - \left(\frac{1}{6}\zeta(2) + \frac{7}{24}\right)\zeta(3) + \frac{191}{144} \zeta (4)+\frac{1}{2}\zeta (5).
    \end{split}
\end{equation}
However, this does not show that the series satisfies Conjecture \ref{FurduiSintConjecture}. Since $\zeta(2)$ and Apéry's constant, $\zeta(3)$, are irrational \cite{Apery}, we deduce that $-\left(\frac{4}{9} + \frac{1}{6}\zeta(3)\right)$ and $- \left(\frac{1}{6}\zeta(2) + \frac{7}{24}\right)$ are irrational. Thus, we cannot conclude \eqref{Zeta2Zeta3Series} shows that our series satisfies Conjecture \ref{FurduiSintConjecture}.

Observe that our series satisfies Conjecture \ref{FurduiSintConjecture} if and only if $\zeta(2)\zeta(3) \in \mathrm{span}_\mathbb{Q}(\mathcal{A})$ where $\mathcal{A} := \{\zeta(2), \zeta(3), \zeta(4), \zeta(5)\}$. A possible way this could happen is if $\zeta(n)$ is a rational multiple of $\pi^n$ for integers $n \geq 2$. Euler himself noticed through his calculations that for integers $n \geq 1$, $\zeta(2n)$ is a rational multiple of $\pi^{2n}$ \cite{Euler2}, which prompted a search for a way to express $\zeta(2n+1)$ as a rational multiple of $\pi^{2n+1}$. If $\zeta(3)$ and $\zeta(5)$ are rational multiples of $\pi^3$ and $\pi^5$ respectively, then $\zeta(2)\zeta(3)$ would be a rational multiple of $\pi^5$ and therefore a rational multiple of $\zeta(5)$. However, there is still no known representation for $\zeta(2n+1)$ for integers $n \geq 1$, even for $\zeta(3)$ \cite[pp. 266-267]{ApostolIntro}. In particular, Euler conjectured that $\zeta(3) = N\pi^3$ where $N$ is an expression involving $\ln(2)$ \cite[p. 12]{Nahin}.

In contemporary literature, mathematicians are more open to the possibility that $\frac{\zeta(3)}{\pi^3}$ is irrational \cite[p. 41]{Finch}. Gun, Murty, and Rath give a conditional proof of the irrationality of $\frac{\zeta(2n+1)}{\pi^{2n+1}}$ for integers $n\geq 1$ in \cite[p. 1341]{ChowlaMilnor}; the proof is contingent on the following conjecture of Chowla and Milnor:
\begin{conjecture}[Chowla--Milnor]\label{ChowlaMilnorConjecture}
    Let $\zeta(s,a)=\sum_{k\geq 0}\frac{1}{(k+a)^s}$ denote the classical Hurwitz zeta function. Then the set
    $$\mathfrak{L}:=\{\zeta(q,m/n):1\leq m\leq n,\text{gcd}(m,n)=1\},\quad (q>1,n>2)$$
    is linearly independent over $\mathbb{Q}.$
\end{conjecture}

In \cite{ChowlaMilnor}, Gun et al. perform a slight modification on $\mathfrak{L}$ by restricting $m<n$ and define the following Chowla--Milnor space:
\begin{align*}
    \mathrm{V}_q(n):=\mathrm{span}_\mathbb{Q}(\{\zeta(q,m/n):1\leq m<n,\text{gcd}(m,n)=1\}).
\end{align*}

Relevant mathematicians notice that the dimension of $\mathrm{V}_q(n)$ as a $\mathbb{Q}$-linear space can shed light on problems regarding irrationality by expressing the Hurwitz zeta function in terms of the Riemann zeta function. It is also worth noting that the Chowla--Milnor conjecture has implications in multiple zeta values. Again in \cite{ChowlaMilnor}, Gun et al. give an important proposition arising from the Chowla--Milnor space.
\begin{proposition}\label{GunProposition}
    The Chowla--Milnor conjecture for $\mathrm{V}_q(4)$ is equivalent to the irrationality of $\frac{\zeta(2n+1)}{\pi^{2n+1}}$ for integers $n\geq 1.$
\end{proposition}
In relation to Theorem \ref{MainProposition}, the falsehood of the Chowla--Milnor conjecture allows for the rationality of certain values of $\frac{\zeta(2n+1)}{\pi^{2n+1}}$, which relates to Conjecture \ref{FurduiSintConjecture} of Furdui and S{\^\i}nt{\u{a}}m{\u{a}}rian. On the other hand, proving the Chowla--Milnor conjecture would only eliminate one possible way of expressing $\zeta(2)\zeta(3)$ as a rational multiple of $\zeta(5)$. More work must be done to establish a stronger implication between these related conjectures. Nevertheless, one may conjecture that $\zeta(2)\zeta(3) \not\in \mathrm{span}_\mathbb{Q} (\mathcal{A})$; however, verifying this is difficult.

Perhaps we can learn more by computing the closed form for
$$\sum_{k=1}^\infty \frac{H_k H_{k+1}\cdots H_{k+q}}{k(k+1)\cdots (k+q)}, \quad (q \in \mathbb{N}).$$

\bibliographystyle{unsrt}
\bibliography{references.bib}

\begin{thebibliography}{10}

\bibitem{Euler1}
Leonhard Euler.
\newblock Meditationes circa singulare serierum genus.
\newblock \url{https://scholarlycommons.pacific.edu/euler-works/477}, 1776.
\newblock Euler Archive - All Works, E477.

\bibitem{SharpMathAnal}
Alina S{\^\i}nt{\u{a}}m{\u{a}}rian and Ovidiu Furdui.
\newblock {\em Sharpening mathematical analysis skills}.
\newblock Springer, New York, 2021.

\bibitem{ExplicitEulerSum}
Junesang Choi and Hari~M Srivastava.
\newblock {Explicit evaluation of Euler and related sums}.
\newblock {\em Ramanujan J.}, 10(1):51--70, 2005.

\bibitem{AssociatedFun}
Hari~M Srivastava and Junesang Choi.
\newblock {\em Series associated with the zeta and related functions}, volume 530.
\newblock Springer Science \& Business Media, 2001.

\bibitem{FlajoletEuler}
Philippe Flajolet and Bruno Salvy.
\newblock Euler sums and contour integral representations.
\newblock {\em Exp. Math.}, 7(1):15--35, 1998.

\bibitem{Sofo}
Anthony Sofo and Amrik~Singh Nimbran.
\newblock {Euler sums and integral connections}.
\newblock {\em Mathematics}, 7(9):833, 2019.

\bibitem{LogIntHarmSeriesBook}
Ali~Shadhar Olaikhan.
\newblock {\em An Introduction To the Harmonic Series And Logarithmic Integrals-For High School Students Up To Researchers}.
\newblock Independent publisher, Phoenix, AZ, 2023.

\bibitem{Xu}
Ce~Xu and Yulin Cai.
\newblock {On harmonic numbers and nonlinear Euler sums}.
\newblock {\em J. Math. Anal. Appl.}, 466(1):1009--1042, 2018.

\bibitem{XuEuler}
Ce~Xu, Yuhuan Yan, and Zhijuan Shi.
\newblock {Euler sums and integrals of polylogarithm functions}.
\newblock {\em J. Number Theory}, 165:84--108, 2016.

\bibitem{XuEulerSums2}
Ce~Xu and Weiping Wang.
\newblock {Two variants of Euler sums}.
\newblock {\em Monatsh. Math.}, 199(2):431--454, 2022.

\bibitem{XuEulerSums3}
Ce~Xu.
\newblock {Explicit evaluations for several variants of Euler sums}.
\newblock {\em Rocky Mt. J. Math.}, 51(3):1089--1106, 2021.

\bibitem{XuEulerSums4}
Ce~Xu and Weiping Wang.
\newblock {Dirichlet type extensions of Euler sums}.
\newblock {\em C. R. Math.}, 361(G6):979--1010, 2023.

\bibitem{CombAlgMZV}
Douglas Bowman and David~M Bradley.
\newblock The algebra and combinatorics of shuffles and multiple zeta values.
\newblock {\em Journal of Combinatorial Theory, Series A}, 97(1):43--61, 2002.

\bibitem{AlgebraMZV}
Michael~E Hoffman.
\newblock Algebraic aspects of multiple zeta values.
\newblock In {\em Zeta functions, topology and quantum physics}, pages 51--73. Springer, 2005.

\bibitem{ZhaoMZV}
Jianqiang Zhao.
\newblock {\em Multiple zeta functions, multiple polylogarithms and their special values}, volume~12.
\newblock World Scientific, 2016.

\bibitem{Borwein}
David Borwein and Jonathan~M Borwein.
\newblock On an intriguing integral and some series related to $\zeta(4)$.
\newblock {\em Proc. Am. Math. Soc.}, 123(4):1191--1198, 1995.

\bibitem{FurduiValean}
Cornel~Ioan V{\u{a}}lean and Ovidiu Furdui.
\newblock {Reviving the quadratic series of Au-Yeung}.
\newblock {\em J. Class. Anal.}, 6(2):113--118, 2015.

\bibitem{Valean3}
Cornel~Ioan V{\u{a}}lean.
\newblock A master theorem of series and an evaluation of a cubic harmonic series.
\newblock {\em J. Class. Anal.}, 10:97--107, 2017.

\bibitem{EulerStirling}
Weiping Wang and Yanhong Lyu.
\newblock {Euler sums and Stirling sums}.
\newblock {\em J. Number Theory}, 185:160--193, 2018.

\bibitem{FurduiProduct2Series}
Ovidiu Furdui.
\newblock Series involving products of two harmonic numbers.
\newblock {\em Math. Mag.}, 84(5):371--377, 2011.

\bibitem{FurduiGazetaMat}
Ovidiu Furdui and Alina S{\^\i}nt{\u{a}}m{\u{a}}rian.
\newblock A series involving a product of three consecutive harmonic numbers.
\newblock {\em Gazeta Matematica, Seria A}, 41(120):35--38, 2023.

\bibitem{Apery}
Roger Ap{\'e}ry.
\newblock Irrationalit{\'e} de $\zeta$ (2) et $\zeta$ (3).
\newblock {\em Ast{\'e}risque}, 61:11--13, 1979.

\bibitem{Euler2}
Leonhard Euler.
\newblock De summis serierum reciprocarum.
\newblock \url{https://scholarlycommons.pacific.edu/euler-works/41}, 1740.
\newblock Euler Archive - All Works, E41.

\bibitem{ApostolIntro}
Tom~M. Apostol.
\newblock {\em Introduction to analytic number theory}.
\newblock Springer Science \& Business Media, 2013.

\bibitem{Nahin}
Paul~J. Nahin.
\newblock {\em In Pursuit of Zeta-3: The World's Most Mysterious Unsolved Math Problem}, chapter~1, pages 1--74.
\newblock Princeton University Press, 2021.

\bibitem{Finch}
Steven~R. Finch.
\newblock {\em {Mathematical constants}}.
\newblock Cambridge university press, Cambridge, England, 2003.

\bibitem{ChowlaMilnor}
Purusottam~Rath Sanoli~Gun, M. Ram~Murty.
\newblock On a conjecture of chowla and milnor.
\newblock {\em Canad. J. Math.}, 63(6):1328--1344, 2011.

\end{thebibliography}

\end{document}